\documentclass[12pt]{article}

\usepackage{amsmath,amsthm,amsfonts,amssymb,graphicx,color,verbatim}

\theoremstyle{plain}
\newtheorem{theorem}{Theorem}[section]

\newtheorem{corollary}[theorem]{Corollary}

\newtheorem{remark}[theorem]{Remark}
\newtheorem{lemma}[theorem]{Lemma}
\newtheorem{definition}[theorem]{Definition}
\DeclareMathOperator*{\argmin}{argmin}
\newcommand{\R}{\mathbb{R}}

\newcommand{\N}{\mathbb{N}}

\newcommand{\abs}[1]{\left\vert#1\right\vert} 
\begin{document}

\title{Quantitative asymptotic regularity results for \\the composition of two mappings}
\author{ U. Kohlenbach$^{a}$, G. L\'opez-Acedo$^{b}$, A. Nicolae$^{c}$
}
\date{}
\maketitle

\begin{center}
{\scriptsize
$^{a}$Department of Mathematics, Technische Universit\" at Darmstadt, Schlossgartenstra\ss{}e 7, 64289 Darmstadt, Germany
\ \\
$^{b}$Department of Mathematical Analysis, University of Seville, Apdo. 1160, 41080-Seville, Spain
\ \\
$^{c}$Department of Mathematics, Babe\c s-Bolyai University, Kog\u alniceanu 1, 400084 Cluj-Napoca, Romania \\
\ \\
E-mail addresses: kohlenbach@mathematik.tu-darmstadt.de (U. Kohlenbach), glopez@us.es (G. L\'{o}pez-Acedo), anicolae@math.ubbcluj.ro (A. Nicolae)
}
\end{center}

\noindent\rule[8 true pt]{\linewidth}{.5 true pt}
\noindent{\footnotesize {\em Abstract:} In this paper, we use techniques which originate from proof mining  to give rates of asymptotic regularity and metastability for a sequence associated to the composition of two firmly nonexpansive mappings.\\



\noindent\rule[8 true pt]{\linewidth}{.5 true pt}\\

%

\section{Introduction} 
This paper continues the work initiated in \cite{AriLopNic15} where the asymptotic behavior of compositions of finitely many firmly nonexpansive mappings was studied with the main focus on asymptotic regularity and convergence results. Since the subdifferential of a proper, convex and lower semi-continuous function is a maximal monotone operator and the resolvent of a monotone operator is firmly nonexpansive, certain splitting methods applied to convex minimization problems are a very relevant instance where compositions of firmly nonexpansive mappings appear. In this line, Bauschke, Combettes and Reich \cite{BauComRei05} proved that if $f$ and $g$ are proper, convex and lower semicontinuous functions defined on a Hilbert space $H$, composing alternatively the resolvents of $f$ and $g$, one obtains weak convergence to a solution (if it exists) of the following minimization problem associated to $f$ and $g$:
\begin{equation}\label{def-fct-min}
\argmin_{(x,y)\in H \times H}  \left(f(x) + g(y) + \frac{1}{2\lambda}\|x-y\|^2\right),
\end{equation}
where $\lambda > 0$.  This problem  covers,  among others,  the convex feasibility problem  for two sets both in the consistent and inconsistent case.
 
 More recently, see \cite{Ban14},   problem (\ref{def-fct-min}) was considered  in the setting of CAT$(0)$ spaces  where it can be approached as in Hilbert spaces by applying alternatively the two resolvents of $f$ and $g$,  $J_\lambda^f$ and $J_\lambda^g$, respectively (which are well-defined in this context). When evaluating the values of the resolvents, errors may also be taken into account. Thus, given a starting point $x_0$, one can construct the sequences $(x_n)$ and $(y_n)$ defined by 
\begin{equation}
d\left(y_n,J_\lambda^g x_n\right) \le \varepsilon_n \quad \text{and} \quad d\left(x_{n+1},J_\lambda^f y_n\right) \le \delta_n,\quad\text{for each }n\in \N,
\end{equation}
where $\displaystyle \sum_{n = 0}^\infty\varepsilon_n < \infty$ and $\displaystyle\sum_{n =0}^\infty\delta_n < \infty$. Note that, if $(x^*,y^*)$ is a solution of \eqref{def-fct-min}, then $y^* = J_\lambda^g(x^*)$ and $x^* = J_\lambda^f(y^*)$, so $\text{Fix}(J_\lambda^f \circ J_\lambda^g) \ne \emptyset$. At the same time, if $x^* \in \text{Fix}(J_\lambda^f \circ J_\lambda^g)$, then $(x^*, J_\lambda^g(x^*))$ is a solution of \eqref{def-fct-min}. Since $J_\lambda^f$ and $J_\lambda^g$ are firmly nonexpansive, by \cite[Theorem 3.3]{AriLopNic15}, it follows that the sequences $(x_n)$ and $(y_n)$ are asymptotically regular (that is, $\displaystyle \lim_{n \to \infty}d(x_n,x_{n+1}) = 0$ and $\displaystyle\lim_{n \to \infty}d(y_n,y_{n+1}) = 0$), provided problem \eqref{def-fct-min} has a solution. If the range of one of the resolvents is boundedly compact, by \cite[Theorem 4.2]{AriLopNic15}, there exists  $u \in {\rm Fix}(J_\lambda^f \circ J_\lambda^g)$ such that $(x_n)$ and $(y_n)$ converge to $u$ and $J_\lambda^g u$, respectively.

In this paper, we use techniques which originate from proof mining (see 
section \ref{proof-mining} below and \cite{Koh08} for more details) to give explicit quantitative forms of these results. Section 3 contains our main result that provides a rate of asymptotic regularity for the sequences $(x_n)$ and $(y_n)$ obtained by composing alternatively two general firmly nonexpansive mappings (with or without errors) in CAT$(0)$ spaces. Section 4 focuses on rates of metastability for these two sequences: based on 
general facts from computability theory one can rule out the existence of 
computable rates of convergence for $(x_n)$ (or for $(y_n)$). However, metastability 
\[ \forall k\in \N\,\forall g:\N\to\N \,\exists n\in\N\, \forall i,j\in 
[n,n+g(n)] \ (d(x_i,x_j)<\frac{1}{k+1}), \] 
though noneffectively equivalent to the full Cauchy property of $(x_n),$ 
does admit (on general logical grounds) effective bounds $\Phi(k,g)$ on 
`$\exists n\in\N$'. We call such a bound $\Phi$ a rate of metastability. 
This concept has been known in logic as the Kreisel `no-counterexample 
interpretation' of which it is a special instance, and for the 
case at hand also coincides with the G\"odel functional interpretation (see 
\cite{Koh08}).
In 2007, the concept was rediscovered by T. Tao (\cite{Tao07}) who introduced the name 
`metastability' for it. Disregarding error terms for the moment, 
in our situation a rate of metastability might be seen as a far reaching 
generalization of a rate of asymptotic regularity as the latter results 
as the special case of the former where $g\equiv 1:$  \\ From 
\[ \forall k\in\N \,\exists n\le \Phi(k,1) \ (d(x_n,x_{n+1})<\frac{1}{k+1}), \]
the fact that $d(x_n,x_{n+1})$ is nonincreasing immediately gives
\[ \forall k\in \N \,\forall n\ge \Phi(k,1) \ (d(x_n,x_{n+1})<\frac{1}{k+1}). \]

\section{Preliminaries}
\subsection{CAT(0) spaces}
Let $(X,d)$ be a metric space. A {\it geodesic path} that joins two points $x,y \in X$ is a mapping $\gamma:[0,l]\subseteq\R\to X$ such that $\gamma (0)=x$, $\gamma(l)=y$ and $d(\gamma(t),\gamma(t'))=\abs{t-t'}$ for all $t,t'\in[0,l]$. The image of $\gamma$ is called a {\it geodesic segment} from $x$ to $y$. A point $z \in X$ belongs to such a geodesic segment if there exists $t\in [0,1]$ such that $d(x,z)=td(x,y)$ and $d(y,z)=(1-t)d(x,y)$ and we write $z=(1-t)x+ty$. $(X,d)$ is a {\it (uniquely) geodesic space} if every two points in $X$ are joined by a (unique) geodesic path. A subset $C$ of $X$ is {\it convex} if it contains all geodesic segments that join any two points in $C$. For more details on geodesic metric spaces, see \cite{BriHae99}.

There are several equivalent conditions for a geodesic metric space $(X, d)$ to be CAT(0), one of them being the following inequality (see, for example, \cite[Theorem 1.3.3]{Bac14}) which is to be satisfied for any four points $x,y,u,v \in X$
\begin{equation}\label{def-CAT0}
d(x,y)^2 + d(u,v)^2 \le d(x,v)^2 + d(y,u)^2 + 2d(x,u)d(y,v).
\end{equation}
CAT$(0)$ spaces include Hilbert spaces, $\mathbb{R}$-trees, Euclidean buildings, complete simply connected Riemannian manifolds of nonpositive sectional curvature, and many other important spaces.

\subsection{Firmly nonexpansive mappings}
Firmly nonexpansive mappings were introduced in Banach spaces by Bruck \cite{Bru73} (in the context of Hilbert spaces, these  mappings are precisely the firmly contractive ones considered earlier by Browder~\cite{Bro67}). Recently, Bruck's definition was extended to a nonlinear setting in~\cite{AriLeuLop14} (see also \cite{Nic13}; in the case of the Hilbert ball this is already due to 
\cite{GoebelReich}). 

\begin{definition}
Let $C$ be a nonempty subset of a CAT$(0)$ space $(X,d)$. A mapping $T:C\to X$ is {\it firmly nonexpansive} if
\[d(Tx,Ty)\leq d((1-\lambda)x + \lambda Tx,(1-\lambda)y + \lambda Ty),\]
for all $x,y\in C$ and $\lambda\in [0,1]$.
\end{definition}

Let $X$ be a complete CAT$(0)$ space. The metric projection onto closed and convex subsets of $X$ is firmly nonexpansive. Another important example of a firmly nonexpansive mapping is the {\it resolvent} of a convex, lower semi-continuous and proper function $f : X \to (-\infty,+\infty]$,
\[J_\lambda^f(x) := \argmin_{z \in X}\left(f(z) + \frac{1}{2\lambda}d(x,z)^2\right),\]
where $\lambda > 0$. 

In CAT$(0)$ spaces, every firmly nonexpansive mapping satisfies the condition below which was called property ($P_2$) in \cite{AriLopNic15}. Moreover, in this setting, every mapping with property ($P_2$) is nonexpansive. Note also that in Hilbert spaces, this notion coincides with firm nonexpansivity.

\begin{definition}
Let $C$ be a nonempty subset of a metric space $(X,d)$. A mapping $T:C\to X$ satisfies {\it property $(P_2)$} if
\begin{equation*}
2d(Tx,Ty)^2\leq d(x,Ty)^2+d(y,Tx)^2-d(x,Tx)^2-d(y,Ty)^2,
\end{equation*}
for all $x,y\in C$.
\end{definition}

\subsection{Proof mining}\label{proof-mining}

During the last two decades a systematic program of `proof mining' has emerged 
as a new applied form of proof theory and has successfully been applied to a number of areas of core mathematics (see \cite{Koh08} for a comprehensive 
treatment up to 2008). 
This logic-based program has its roots in Georg Kreisel's pioneering ideas of 
`unwinding of proofs' going back to the 1950's and is concerned with the 
extraction of explicit effective bounds from prima facie noneffective proofs. 
General logical metatheorems guarantee such extractions for large classes of 
proofs and provide algorithms (based on so-called proof interpretations) 
for the actual extraction from a given proof. 
This approach has been applied with particular success in the context of 
nonlinear analysis 
including fixed point theory, ergodic theory, topological dynamics, 
continuous optimization 
and abstract Cauchy problems. 
\\[1mm] 
One condition that 
guarantees such results is that the statement proven has (if written 
in the appropriate formal framework) the form 
\[ \forall \underline{x}\in \underline{X}^{(\underline{X})}\, \exists n\in\N \,
A_{\exists}(\underline{x},n),\] 
where $\underline{x}$ is a tuple of parameters ranging over various 
metric, hyperbolic or normed spaces $\underline{X}$ or suitable classes 
of mappings between such spaces and $A_{\exists}$ 
is purely existential. This is not the case for the usual formulation 
of the Cauchy property which is of the form $\forall\exists\forall$ but 
is satisfied for the (equivalent) metastable formulation since the bounded 
quantifier $\forall i,j\in [n,n+g(n)]$ can be disregarded. \\ However, for 
asymptotic regularity results $d(x_n,x_{n+1})\to 0$ 
(rather than the convergence of $(x_n)$  
itself) one usually can obtain full rates of convergence. One 
reason for 
this is that the sequence $(d(x_n,x_{n+1}))_{n\in\N}$ often is nondecreasing 
(as is the case for the sequences $(x_n),(y_n)$ defined by (\ref{it:one}) 
below). Then $d(x_n,x_{n+1})\to 0$ is equivalent to 
\[ \forall k\in\N\,\exists n\in\N \, (d(x_n,x_{n+1})<\frac{1}{k+1}), \] 
which has the right logical form and any bound $\Phi(k)$ on $n$ is a 
rate of asymptotic regularity.
\\[1mm] In the next section we will present rates of asymptotic regularity 
that have been extracted using this methodology from a noneffective 
asymptotic regularity proof given in \cite{AriLopNic15}. The noneffectivity 
of that proofs comes from the (repeated) use of the convergence of bounded 
monotone sequences which is known to fail in a computable reading. The 
analysis of the proof was obtained by first replacing the use of the limits 
of such sequences by sufficiently good Cauchy-points instead and then applying 
a well-known and effective rate of metastability for the Cauchy 
property of bounded monotone sequences to suitably chosen parameters 
$g$ and $\varepsilon$ 
(see the end of the proof of Theorem \ref{thm-as-reg-P2} below). \\[1mm] 
In the final 
section we will apply an effective proof mining result from \cite{KohLeuNic15} to convert 
the rates of asymptotic regularity into rates of metastability provided 
that the underlying space is compact. 
\\[1mm] As usual in applications of the proof mining methodology, the final 
bounds 
and the proofs of their correctness 
can be stated in ordinary mathematical terms without 
any reference to tools or concepts from logic.
\section{Rate of asymptotic regularity}

Let $X$ be a metric space, $T_1,T_2:X\to X$ and $(x_n)$ and $(y_n)$ be defined by
\begin{equation}\label{it:one}
y_n := T_1 x_n \quad \text{and} \quad x_{n+1} := T_2 y_n,\quad\text{for each }n\in \mathbb{N}.
\end{equation}

\begin{theorem} \label{thm-as-reg-P2}
Let $(X,d)$ be a CAT$(0)$ space and let $T_1,T_2:X\to X$ satisfy property $(P_2)$. Denote $S :=T_2 \circ T_1$ and suppose that $\emph{Fix}(S)\neq\emptyset$. Let $x_0 \in X$ and $b > 0$ such that there exists $u \in \emph{Fix}(S)$ with $2d(x_0,u) \le b$. Define the sequences $(x_n)$ and $(y_n)$ by~\eqref{it:one}. Then 
\[\forall \varepsilon > 0 \, \forall n \ge \Phi(\varepsilon,b)\, (d(y_n,y_{n+1}) \le d(x_n,x_{n+1}) \le \varepsilon),\]
where
\[
\Phi(\varepsilon,b) := k\left\lceil\frac{2b(1+2^k)}{\varepsilon}\right\rceil +1, \quad \text{with } k := \left\lceil\frac{2b}{\varepsilon}\right\rceil.
\] 
\end{theorem}
\begin{proof}
Let $n \in \mathbb{N}$ and $k \in \mathbb{N}^*$. Since $T_1$ and $T_2$ satisfy property $(P_2)$ we have that

\begin{align*} 2d(y_{n+k+1},y_{n+1})^2 & \le  d(x_{n+k+1},y_{n+1})^2 + d(x_{n+1},y_{n+k+1})^2 \\
& \phantom{a\,}- d(x_{n+k+1},y_{n+k+1})^2 - d(x_{n+1},y_{n+1})^2
\end{align*}
and 
\begin{align*} 2d(x_{n+k+1},x_{n+1})^2 & \le d(y_{n+k},x_{n+1})^2 + d(y_{n},x_{n+k+1})^2\\
& \phantom{a\,} - d(y_{n+k},x_{n+k+1})^2 - d(y_{n},x_{n+1})^2.\end{align*}
These inequalities together with \eqref{def-CAT0} yield
\begin{align*} d(y_{n+k+1},y_{n+1})^2 + d(x_{n+k+1},x_{n+1})^2 & \le d(x_{n+1},x_{n+k+1})d(y_{n+1},y_{n+k})\\ & \phantom{a\,}+d(x_{n+1},x_{n+k+1})d(y_n,y_{n+k+1}).
\end{align*}
Because $$d(y_{n+k+1},y_{n+1})^2 + d(x_{n+k+1},x_{n+1})^2 \ge 2d(y_{n+k+1},y_{n+1})d(x_{n+k+1},x_{n+1})$$ we obtain that
\begin{align*} 2d(y_{n+k+1},y_{n+1})d(x_{n+k+1},x_{n+1}) &\le d(x_{n+1},x_{n+k+1})d(y_{n+1},y_{n+k})\\
& \phantom{a\,} +d(x_{n+1},x_{n+k+1})d(y_n,y_{n+k+1}).\end{align*}
Hence,
\[2d(y_{n+k+1},y_{n+1}) - d(y_{n+1},y_{n+k}) \le d(y_n,y_{n+k+1})\]
(note that the above relation is also true when $d(x_{n+k+1},x_{n+1}) = 0$). Since
\[d(y_{n+1},y_{n+k}) \le d(y_n,y_{n+k-1}) \le d(y_n,y_{n+1}) + d(y_{n+1},y_{n+k-1}),\]
it follows that $d(y_{n+1},y_{n+k}) \le (k-1) d(y_n,y_{n+1})$. Denote $$r_{n,k} := d(y_n,y_{n+k}) \le d(x_n, x_{n+k}) \le 2 d(x_0,u) \le b.$$ Then
\[2r_{n+1,k} - (k-1)r_{n,1} \le r_{n,k+1}.\]
We show next that for any $n \in \mathbb{N}$ and $k \in \mathbb{N}^*$, 
\begin{equation}\label{eq1}
r_{n,k} \ge k r_{n+k,1} - k 2^k (r_{n,1} - r_{n+k,1}).
\end{equation} 
For $k=1$ this relation obviously holds for every $n \in \mathbb{N}$. Suppose that \eqref{eq1} holds for each $n \in \mathbb{N}$. Then for each $n \in \mathbb{N}$,
\begin{align*}
r_{n,k+1}  & \ge 2r_{n+1,k} - (k-1)r_{n,1} \\ &\ge 2\left(k r_{n+1+k,1} - k 2^k (r_{n+1,1} - r_{n+1+k,1})\right) - (k-1)r_{n,1}\\
& \ge  2\left(k r_{n+1+k,1} - k 2^k (r_{n,1} - r_{n+1+k,1})\right) - (k-1)r_{n,1}\\
& = (k+1) r_{n+1+k,1} - k 2^{k+1} (r_{n,1} - r_{n+1+k,1}) - (k-1)r_{n,1} + (k-1) r_{n+1+k,1}\\
& = (k+1) r_{n+1+k,1} - (k 2^{k+1} + k - 1)(r_{n,1} - r_{n+1+k,1})\\
& \ge (k+1) r_{n+1+k,1} - (k+1)2^{k+1}(r_{n,1} - r_{n+1+k,1}).
\end{align*}
Thus \eqref{eq1} holds for every $n \in \N$ and $k \in \N^*$, from where
\[r_{n,1} - (1+2^k)(r_{n,1}-r_{n+k,1}) \le \frac{r_{n,k}}{k} \le \frac{b}{k}.\]
Let $\varepsilon > 0$ and $\displaystyle k:=\left\lceil\frac{2b}{\varepsilon}\right\rceil$. Then for all $n \in \mathbb{N}$,
\[r_{n,1} - (1+2^k)(r_{n,1}-r_{n+k,1}) \le \frac{\varepsilon}{2}.\]
Note that $(r_{n,1})$ is a nonincreasing sequence bounded by $b$. Using \cite[Proposition 2.27]{Koh08} for $\displaystyle\varepsilon' := \frac{\varepsilon}{2(1+2^k)}$ and $g\equiv k$, there exists $\displaystyle N \le k\left\lceil \frac{2b(1+2^k)}{\varepsilon} \right\rceil$ such that $\displaystyle r_{N,1}-r_{N+k,1}  \le \frac{\varepsilon}{2(1+2^k)}$ and so $r_{N,1} \le \varepsilon$. Thus, for $n \ge \Phi(\varepsilon,b)$,
\[d(y_n,y_{n+1}) \le d(x_n,x_{n+1}) \le d(y_{n-1},y_{n}) \le r_{N,1} \le \varepsilon.\] 
\end{proof}

Suppose now that $(x_n)$ and $(y_n)$ are defined by
\begin{equation}\label{it:two}
d(y_n,T_1 x_n) \le \varepsilon_n \quad \text{and} \quad d(x_{n+1},T_2 y_n) \le \delta_n,\quad\text{for each }n\in \mathbb{N},
\end{equation}
where $\displaystyle \sum_{n = 0}^\infty \varepsilon_n< \infty$ and $\displaystyle \sum_{n = 0}^\infty \delta_n< \infty$. For $n \in \mathbb{N}$, denote $\gamma_n := \varepsilon_n + \delta_n$. 

\begin{lemma} \label{lemma1}
Let $(X,d)$ be a CAT$(0)$ space and let $T_1,T_2:X\to X$ satisfy property $(P_2)$. Denote $S:=T_2 \circ T_1$. Then for every $n \in \mathbb{N},$ 
\begin{itemize}
\item[(i)] $d(x_{n+1},S x_{n}) \le \gamma_{n}.$
\item[(ii)] If $u \in \emph{Fix}(S)$, $d(x_{n+1},u) \le \gamma_n + d(x_n,u)$.
\end{itemize}
\end{lemma}
\begin{proof}
\hskip 18cm

(i) $d(x_{n+1},S x_{n}) \le d(x_{n+1}, T_2 y_{n}) + d(T_2 y_{n}, T_2(T_1 x_{n})) \le \delta_{n} + d(y_{n},T_1 x_{n}) \le \gamma_{n}.$\\

(ii) Let $u \in \text{Fix}(S)$. Then
\begin{align*}
d(x_{n+1},u) &\le d(x_{n+1},T_2 y_n) + d(T_2 y_n, u) \le \delta_n + d(y_n, T_1 u)\\
& \le \delta_n + d(y_n, T_1 x_n) + d(T_1 x_n, T_1 u) \le \gamma_n + d(x_n,u).
\end{align*}
\end{proof}

\begin{theorem} \label{thm-as-reg-P2-errors-simple}
Let $(X,d)$ be a CAT$(0)$ space and let $T_1,T_2:X\to X$ satisfy property $(P_2)$ with $\emph{Fix}(S)\neq\emptyset,$ where $S:=T_2\circ T_1.$ Let $x_0 \in X$. Define the sequences $(x_n)$ and $(y_n)$ by~\eqref{it:two}. Suppose that $\displaystyle\sum_{n=0}^\infty \gamma_n$ converges with Cauchy modulus $\alpha$, i.e., $\alpha : (0,\infty) \to \mathbb{N}$,
\[\forall \varepsilon > 0 \, \forall k \in \mathbb{N}\,\left(\sum_{i=\alpha(\varepsilon)}^{\alpha(\varepsilon) + k} \gamma_i \le \varepsilon\right).\]
Let $B \ge 0$ such that $\displaystyle\sum_{n\ge 0}\gamma_n \le B $ and $b > 0$ such that there exists $u \in \emph{Fix}(S)$ with $d(x_0,u) \le b.$ Then 
\[\forall \varepsilon > 0 \, \forall n \ge \Phi'(\varepsilon,b, B, \alpha) \, (d(x_n,x_{n+1}) \le \varepsilon),\]
where
\[
\Phi'(\varepsilon,b,B,\alpha) := \alpha(\varepsilon/3) + \Phi(\varepsilon/3,2(b+B)),
\] 
and $\Phi$ is the rate of asymptotic regularity from Theorem 
\ref{thm-as-reg-P2}. 
For $d(y_n,y_{n+1})$ the same result holds with rate 
$\Phi''(\varepsilon,b,B,\alpha):= \Phi'(\varepsilon/2,b,B,\alpha).$
\end{theorem}
\begin{proof} Let $(x_n),(y_n)$ be defined as in \eqref{it:two} and consider 
for $S:=T_2\circ T_1$, 
\[ z_n:= S^n(x_{\alpha(\varepsilon/3)}). \] 
Note that $(z_n)$ is the sequence 
$(x_n)$ defined by \eqref{it:one} with the starting point 
$$z_0=x_{\alpha(\varepsilon/3)}.$$ \\
By induction, Lemma \ref{lemma1}.(ii) gives 
\[ d(x_n,u)\le d(x_0,u)+\sum^{n-1}_{i=0} \gamma_i.\]
Thus,
\[d(z_0,u)=d(x_{\alpha(\varepsilon/3)},u)\le d(x_0,u)+
\sum^{\alpha(\varepsilon/3)-1}_{i=0} \gamma_i\le b+B\]
and so one can apply Theorem \ref{thm-as-reg-P2} to the sequence $(z_n)$ to get that for every $n\ge \Phi(\varepsilon/3,2(b+B))$, \[d(z_n,z_{n+1})\le \frac{\varepsilon}{3}.\] 
One easily shows by induction on $n$ that for all $n\in\mathbb{N}$,
\[ d(z_n,x_{\alpha(\varepsilon/3)+n})\le 
\sum^{\alpha(\varepsilon/3)+n-1}_{i=\alpha(\varepsilon/3)} \gamma_i \le \frac{\varepsilon}{3}.\]
For $n=0$ this is obvious and for the induction step we argue 
(using Lemma \ref{lemma1}.(i))
\begin{align*}
d(z_{n+1},x_{\alpha(\varepsilon/3)+n+1}) & \le
d(Sz_n,Sx_{\alpha(\varepsilon/3)+n})+d(Sx_{\alpha(\varepsilon/3)+n},
x_{\alpha(\varepsilon/3)+n+1})\\ 
& \le d(z_n,x_{\alpha(\varepsilon/3)+n})
+\gamma_{\alpha(\varepsilon/3)+n}.
\end{align*}
Hence for all $n\ge\alpha(\varepsilon/3)+\Phi(\varepsilon/3,2(b+B))$
\[ d(x_n,x_{n+1})\le d(x_n,z_{n-\alpha(\varepsilon/3)})+
d(z_{n-\alpha(\varepsilon/3)},z_{n-\alpha(\varepsilon/3)+1})+
d(z_{n-\alpha(\varepsilon/3)+1},x_{n+1})\le \varepsilon. \]
The claim for $(y_n)$ follows from 
\[ d(y_n,y_{n+1})\le d(T_1x_n,T_1x_{n+1})+\varepsilon_n+\varepsilon_{n+1}
\le d(x_n,x_{n+1})+\varepsilon_n+\varepsilon_{n+1} \] 
and the fact that for $n\ge \Phi'(\frac{\varepsilon}{2},b,B,\alpha)\ge 
\alpha(\varepsilon/6)$ one has $\varepsilon_n+\varepsilon_{n+1}\le 
\frac{\varepsilon}{6}$ and so 
\[ d(y_n,y_{n+1})\le d(x_n,x_{n+1})+\frac{\varepsilon}{6}\le \frac{\varepsilon}{2}+\frac{\varepsilon}{6}<\varepsilon. \] 
\end{proof}
\begin{remark} In the situation of Corollary 3.3 from {\rm \cite{AriLopNic15}} 
we can take the quadratic rate $\theta$ from that corollary instead of 
the exponential rate $\Phi$ in Theorem \ref{thm-as-reg-P2-errors-simple} 
above.
\end{remark}

\section{Rate of metastability}
Consider the sequences $(x_n),(y_n)$ from (\ref{it:two}).

\begin{lemma} \label{lemma-Fejer}
$\chi(n,m,r):=m(r+1)$ is a modulus of uniform quasi-Fej\'er monotonicity 
(in the sense of {\rm \cite{KohLeuNic15}}) 
of $(x_n)$ w.r.t. $F:=\emph{Fix}(S),$ where $S:=T_2\circ T_1,$ and the error 
terms $\gamma_n:=\delta_n+\varepsilon_n,$ i.e. 
\[ \begin{array}{l} \forall r,n,m\in\mathbb{N} \,\forall p\in X \ \Big( d(p,Sp)\le \frac{1}{\chi(n,m,r)+1} \rightarrow \\ \hspace*{1cm} 
\forall l\le m \left( d(x_{n+l},p)<d(x_n,p)+\sum^{n+l-1}_{i=n} 
\gamma_i +\frac{1}{r+1}\right)\Big). \end{array}\] 
Analogously for $(y_n)$ with $\gamma'_n:=\varepsilon_{n+1}+\delta_n$ and 
$S':=T_1\circ T_2.$
\end{lemma} 
\begin{proof} An easy calculation (see $(14)$ in \cite{AriLopNic15}) gives
that for all $p\in X$ and $n\in\mathbb{N}$ 
\[ d(x_{n+1},Sp)\le d(x_{n},p)+\gamma_n \] 
and so 
\[ d(x_{n+1},p)\le d(x_n,p)+\gamma_n +d(p,Sp). \] 
Hence, 
\[ d(x_{n+l},p)\le d(x_n,p)+\sum^{n+l-1}_{i=n}\gamma_i+ l\cdot d(p,Sp). \]
The claim is now immediate. \\ 
The second claim for $(y_n)$ is proved analogously using 
\[ d(y_{n+1},S'p)\le d(y_n,p)+\gamma'_n. \]
\end{proof}  

\begin{lemma} 
Under the assumptions of 
Theorem \ref{thm-as-reg-P2-errors-simple}, 
let $\beta$ be a rate of convergence for $\gamma_n\to 0$ and 
$\Phi'(\varepsilon):=\Phi'(\varepsilon,b,B,\alpha)$ be the rate of 
convergence for $d(x_n,x_{n+1})\to 0$ from Theorem \ref{thm-as-reg-P2-errors-simple}.
Then $\Phi_{\beta}(\varepsilon):=\max\{ \beta(\varepsilon/2),\Phi'(\varepsilon/2)\}$ 
is a rate of asymptotic regularity for $d(x_n,Sx_n)\to 0,$ i.e. 
\[ \forall \varepsilon>0\,\forall n\ge \Phi_{\beta}(\varepsilon)\, (d(x_n,Sx_n)
\le \varepsilon).\] 
Analogously for $(y_n)$ and $S':=T_1\circ T_2$ with $\beta$ being replaced by a rate of convergence 
$\beta'$ for $\gamma'_n\to 0$ and $\Phi'_{\beta'}(\varepsilon):=
\max\{ \beta'(\varepsilon/2),\Phi''(\varepsilon/2)\}.$ 
\end{lemma}
\begin{proof} Let $n\ge \Phi_{\beta}(\varepsilon).$ Then (using Lemma \ref{lemma1}.(i))
\[ d(x_n,Sx_n) \le d(x_n,x_{n+1})+d(x_{n+1},Sx_n)\le 
\frac{\varepsilon}{2} +\gamma_n \le \frac{\varepsilon}{2}+
\frac{\varepsilon}{2}=\varepsilon. \]
For $(y_n)$ one reasons analogously. Thus, for all $n\ge \Phi'_{\beta'}(\varepsilon) = \max\{ \beta'(\varepsilon/2),\Phi'(\varepsilon/4)\}$,
\[ d(y_n,S'y_n)\le d(y_n,y_{n+1})
+d(y_{n+1},S'y_n)\le \frac{\varepsilon}{2}+\gamma'_n\le \varepsilon.\] 
\end{proof}

\begin{corollary} \label{cor-liminf} 
$\widehat{\Phi}_{\beta}(k,N):=\max\{ N,\max\{ \Phi_{\beta}(1/(i+1)):i\le k\} \}$ 
is a monotone $\liminf$-bound
for $(x_n)$ w.r.t. $\emph{Fix}(S)$ 
in the sense of {\rm \cite{KohLeuNic15}}.  
Analogously for $(y_n)$ with $\Phi'_{\beta'}$ and $\emph{Fix}(S')$. 
\end{corollary} 
\begin{definition}[see \cite{KohLeuNic15}] 
Let $(X,d)$ be a totally bounded metric space. We call a function 
$\gamma:\mathbb{N}\to\mathbb{N}$ a modulus of total boundedness for 
$X$ if for any sequence $(a_n)$ in $X$ 
\[ \exists 0\le i<j\le \gamma(k) \ \left( d(a_i,a_j)\le\frac{1}{k+1}\right). \]
\end{definition}
  
\begin{theorem} 
Under the assumptions of 
Theorem \ref{thm-as-reg-P2-errors-simple}, let $X$ additionally 
be totally bounded with a modulus of total 
boundedness $\gamma.$
Then $(x_n)$ is Cauchy with rate of metastability 
$\widehat{\Psi}(k,g,\gamma,\alpha,\beta,b,B),$ i.e.  for all $ k\in\N$  and for all  $g:\N\to\N $ we have 
\[ \exists n\le 
\widehat{\Psi}(k,g,\gamma,\alpha,\beta,b,B) \,\forall i,j\in 
[n,n+g(n)] \ \left( d(x_i,x_j)\le \frac{1}{k+1}\right), \]
where 
\[ \begin{array}{l} \widehat{\Psi}(k,g,\gamma,\alpha,\beta,b,B):=
\widehat{\Psi}_0(P), \ P:=\gamma (8k+7)+1, \ \xi(k):=\alpha(1/(k+1)), \\[1mm] 
 \chi^M_g(n,k):=(\max\limits_{i\le n} g(i))\cdot (k+1),\end{array}\]
and $\widehat{\Psi}_0$ is defined recursively
\[  \widehat{\Psi}_0(0):=0, \
\widehat{\Psi}_0(n+1):= \widehat{\Phi}_{\beta}\left(\chi^M_g\left(\widehat{\Psi}_0(n),
8k+7\right),\xi(8k+7)\right).  \]  
Analogously for $(y_n)$ with $\Phi'_{\beta'}$ and $\xi(k):=\alpha'(1/(k+1)),$ 
where $\alpha'$ is a Cauchy modulus for $\sum\limits^{\infty}_{n=0} \gamma'_n.$
\end{theorem}
\begin{proof} The result follows from Theorem 6.4 in 
\cite{KohLeuNic15} together with Lemma \ref{lemma-Fejer} and 
Corollary \ref{cor-liminf}. Note that in our case $G=H=id$ and so we can 
take $\alpha_G:=\beta_H:=id$ as well. 
\end{proof}
\section*{Acknowledgements} U. Kohlenbach has been supported by the German Science Foundation (DFG
Project KO 1737/5-2).  G. L\'{o}pez-Acedo has been  supported by DGES (Grant MTM2012-34847-C02-01), Junta de Andaluc\'{\i}a (Grant P08-FQM-03453). A.~Nicolae has been supported by a grant of the Babe\c{s}-Bolyai University, project number GSCE-30256-2015. Part of this work was carried out while U. Kohlenbach and   A. Nicolae were visiting the University of Seville. They would like to thank the Department of Mathematical Analysis and the Institute of Mathematics of the University of Seville (IMUS) for the hospitality.

\end{document}